\documentclass[11pt]{article}

\usepackage[T1]{fontenc}
\usepackage[utf8]{inputenc}
\usepackage{lmodern}
\usepackage{amsmath,amssymb,amsthm,mathtools}
\usepackage[letterpaper,margin=1in]{geometry}
\usepackage{microtype}
\usepackage{enumitem}
\usepackage{indentfirst}

\usepackage{xcolor}
\usepackage{tikz}

\usepackage[
pdfauthor={Songling Shan and Yucheng Zhong},
pdftitle={Total Vertex Irregularity Strength of Cubic and 4-Regular Graphs},
pdfstartview=XYZ,
bookmarks=true,
colorlinks=true,
linkcolor=blue,
urlcolor=blue,
citecolor=blue,
linktocpage=true,
hyperindex=true
]{hyperref}

\tikzset{
	tvvertex/.style={circle,draw=black,fill=white,text=black,
		minimum size=5.2mm,inner sep=0pt,font=\small},
	tvedge/.style={draw=black!45,line width=.45pt},
	tvheavy/.style={draw=black,line width=1.15pt},
	tvweight/.style={font=\scriptsize,text=black,inner sep=2pt}
}
\newcommand{\tvvertex}[4]{%
	\node[tvvertex,label={[tvweight]#4:{$[#3]$}}] at (#1) {$#2$};}

\newtheorem{theorem}{Theorem}[section]
\newtheorem{lemma}[theorem]{Lemma}
\newtheorem{claim}[theorem]{Claim}

\newtheorem{conjecture}[theorem]{Conjecture}

\newtheorem*{externallemma}{Lemma}
\theoremstyle{remark}

\numberwithin{equation}{section}

\DeclareMathOperator{\tvs}{tvs}
\DeclareMathOperator{\wt}{wt}
\newcommand{\ceil}[1]{\left\lceil #1\right\rceil}
\newcommand{\floor}[1]{\left\lfloor #1\right\rfloor}
\newcommand{\defect}{\operatorname{def}}

\newenvironment{claimproof}[1][Proof]{\begin{proof}[#1]}{\end{proof}}

\begin{document}
	\title{Total Vertex Irregularity Strength of Cubic and $4$-Regular Graphs}
	\author{Songling Shan\footnote{Auburn University, Department of Mathematics and Statistics, Auburn, AL 36849.
			Email: \texttt{szs0398@auburn.edu}.
			Supported in part by NSF grant DMS-2451895.}
		\qquad
		Yucheng Zhong\footnote{Auburn University, Department of Mathematics and Statistics, Auburn, AL 36849.
			Email:	\texttt{yzz0237@auburn.edu}. }
	}

	\date{\today}
	\maketitle

	\begin{abstract}
		Let $G$ be a graph and let $k$ be a positive integer. A total $k$-labeling of $G$ assigns to each vertex and each edge a label from $\{1,\ldots,k\}$. The weight of a vertex is the sum of its label and the labels of its incident edges. A total labeling is vertex irregular if all vertex weights are distinct. The total vertex irregularity strength $\tvs(G)$ is the smallest $k$ for which $G$ has a vertex irregular total $k$-labeling. For a $d$-regular graph $G$ on $n$ vertices, a counting argument gives $\tvs(G)\ge\lceil(n+d)/(d+1)\rceil$. The restriction of a conjecture of Nurdin, Baskoro, Salman, and Gaos to regular graphs asserts that this bound is attained. We prove this assertion for cubic and $4$-regular graphs. We also show that, for every fixed $d\ge2$, a recent theorem on prescribed degree frequencies implies the assertion for all sufficiently large $d$-regular graphs.
	\end{abstract}

	\medskip
	\noindent\textbf{Keywords.} Graph labeling; Total vertex irregularity strength; Matching.

	\section{Introduction}\label{sec:intro}

	All graphs considered in this paper are finite and simple, and need not be connected.  Let $G$ be a graph.  Denote by $V(G)$ and $E(G)$ the vertex set and edge set of $G$, respectively.  For $v\in V(G)$, $N_G(v)$ denotes the set of neighbors of $v$ in $G$, and $d_G(v):=|N_G(v)|$ is the degree of $v$ in $G$.  We denote by $\delta(G)$ and $\Delta(G)$ the minimum and maximum degree of $G$, respectively.  For a graph $H$ and an integer $i\ge0$, let $V_i(H)=\{v\in V(H):d_H(v)=i\}$ and $n_i(H)=|V_i(H)|$.  For $S\subseteq V(G)$, the subgraph of $G$ induced by $S$ is denoted by $G[S]$, and $G-S:=G[V(G)\setminus S]$.  We write $G-x$ for $G-\{x\}$.  If $F\subseteq E(G)$, then $G-F$ is obtained from $G$ by deleting all the edges of $F$.  For a graph $H$ and a positive integer $k$, $kH$ denotes the disjoint union of $k$ copies of $H$.  For two integers $p$ and $q$, let $[p,q]=\{i\in\mathbb{Z}: p\le i\le q\}$.

	A \emph{total $k$-labeling} of $G$ is a mapping $\lambda:V(G)\cup E(G)\to[1,k]$.  For $v\in V(G)$, the \emph{weight} of $v$ under $\lambda$ is
	\[
	\wt_{\lambda}(v)=\lambda(v)+\sum_{e\ni v}\lambda(e).
	\]
	We say that $\lambda$ is \emph{vertex irregular} if $\wt_{\lambda}(u)\ne\wt_{\lambda}(v)$ for any two distinct vertices $u,v\in V(G)$.  The \emph{total vertex irregularity strength} of $G$, denoted $\tvs(G)$, is the smallest integer $k$ for which $G$ has a vertex irregular total $k$-labeling.  This parameter was introduced by Ba\v{c}a, Jendro\v{l}, Miller, and Ryan~\cite{BJMR}.

	Suppose that $G$ is $d$-regular on $n$ vertices.  Under a total $k$-labeling of $G$, the weight of a vertex is a sum of $d+1$ labels and so lies in $[d+1,(d+1)k]$.  Since the $n$ vertex weights are distinct in a vertex irregular labeling, we have $(d+1)k-d\ge n$, that is,
	\begin{equation}\label{eq:regular-lower}
		\tvs(G)\ge \ceil{\frac{n+d}{d+1}}.
	\end{equation}
	The same argument gives a lower bound for an arbitrary graph $G$ with minimum degree $\delta$ and maximum degree $\Delta$. For each $i\in[\delta,\Delta]$, the $\sum_{j=\delta}^{i}n_j(G)$ vertices of degree at most $i$ have weights in $[\delta+1,(i+1)k]$ under a total $k$-labeling. Hence
	\[
	\tvs(G)\ge\ceil{\frac{\delta+\sum_{j=\delta}^{i}n_j(G)}{i+1}}.
	\]
	Nurdin, Baskoro, Salman, and Gaos~\cite{NBSG} conjectured that the largest of these bounds is always attained.

	\begin{conjecture}[Nurdin, Baskoro, Salman, and Gaos~\cite{NBSG}]\label{conj:NBSG}
		For every graph $G$ with minimum degree $\delta$ and maximum degree $\Delta$,
		\[
		\tvs(G)=\max_{\delta\le i\le\Delta}\ceil{\frac{\delta+\sum_{j=\delta}^{i}n_j(G)}{i+1}}.
		\]
	\end{conjecture}

	For a $d$-regular graph on $n$ vertices, Conjecture~\ref{conj:NBSG} asserts equality in~\eqref{eq:regular-lower}. The conjecture is false for general graphs. Susanto, Simanjuntak, and Baskoro~\cite{SSB} constructed infinite families of counterexamples for which $\tvs(G)$ exceeds the conjectured value by exactly one. Their constructions do not give counterexamples for regular graphs, and the restriction of Conjecture~\ref{conj:NBSG} to regular graphs remains open.

	For cubic graphs, Conjecture~\ref{conj:NBSG} asserts that $\tvs(G)=\ceil{(n+3)/4}$. Barra and Afifurrahman~\cite{BA} referred to this statement as a folklore conjecture and proved it for cubic graphs with a perfect matching. Using a recent degree-balanced decomposition theorem of Lu\v{z}ar, Przyby\l{}o, and Sot\'ak~\cite{LPS}, we remove the perfect matching assumption.

	\begin{theorem}\label{thm:cubic}
		Let $G$ be a cubic graph on $n$ vertices, and let $s=\ceil{(n+3)/4}$. Then $\tvs(G)=s$. Moreover, $G$ has a vertex irregular total $s$-labeling in which every edge receives label $1$ or $s$.
	\end{theorem}

	Our next result is the $4$-regular case.

	\begin{theorem}\label{thm:four-main}
		Let $G$ be a $4$-regular graph on $n$ vertices, and let $s=\ceil{(n+4)/5}$.  Then $\tvs(G)=s$.  Moreover, $G$ has a vertex irregular total $s$-labeling whose set of vertex weights is exactly $[5,n+4]$.
	\end{theorem}

	In Section~\ref{sec:cubic}, we prove Theorem~\ref{thm:cubic}. Section~\ref{sec:four} is devoted to the proof of Theorem~\ref{thm:four-main}. In Section~\ref{sec:remarks}, we apply a recent theorem of Cao, Tang, and Wu~\cite{CTW} to show that, for every fixed $d\ge2$, the bound in~\eqref{eq:regular-lower} is attained by all sufficiently large $d$-regular graphs. We conclude by stating the conjecture for regular graphs.

	\section{Proof of Theorem~\ref{thm:cubic}}\label{sec:cubic}

	We use the following theorem of Lu\v{z}ar, Przyby\l{}o, and Sot\'ak~\cite[Theorem~1.4]{LPS}.

	\begin{theorem}[Lu\v{z}ar, Przyby\l{}o, and Sot\'ak~\cite{LPS}]\label{thm:LPS}
		Let $G$ be a cubic graph on $n$ vertices.  If $G$ is not isomorphic to $K_4$, $K_{3,3}$, or $3K_4$, then $G$ has a spanning subgraph $H$ with $n_i(H)\in\{\floor{\frac n4},\ceil{\frac n4}\}$ for each $i\in[0,3]$.
	\end{theorem}

	The next lemma converts a spanning subgraph with small degree classes into a vertex irregular total labeling.

	\begin{lemma}\label{lem:cubic-extension}
		Let $G$ be a cubic graph and $s\ge2$ be an integer.  If $G$ has a spanning subgraph $H$ with $n_i(H)\le s-1$ for each $i\in[0,3]$, then $G$ has a vertex irregular total $s$-labeling in which every edge receives a label in $\{1,s\}$.
	\end{lemma}

	\begin{proof}
		Define $\lambda$ on $E(G)$ by letting $\lambda(e)=s$ if $e\in E(H)$ and $\lambda(e)=1$ otherwise.  For each $i\in[0,3]$, let $V_i(H)=\{v_{i,1},\ldots,v_{i,n_i(H)}\}$ and let $\lambda(v_{i,j})=j$ for each $j\in[1,n_i(H)]$.  A vertex $v_{i,j}$ is incident with $i$ edges of $H$ and $3-i$ edges of $G-E(H)$, so $\wt_{\lambda}(v_{i,j})=is+(3-i)+j=3+i(s-1)+j$.  For a fixed $i$, these weights are distinct and, as $j\le n_i(H)\le s-1$, they lie in $[4+i(s-1),\,3+(i+1)(s-1)]$.  These four intervals, for $i\in[0,3]$, are pairwise disjoint.  Hence $\lambda$ is vertex irregular, as desired.
	\end{proof}

	\begin{figure}[htbp]
		\centering
		\begin{minipage}[t]{.44\textwidth}
			\centering
			\textbf{(a)} $K_4$, $s=2$\par\medskip
			\begin{tikzpicture}
				\coordinate (a) at (-.95,.95);
				\coordinate (b) at (.95,.95);
				\coordinate (c) at (.95,-.95);
				\coordinate (d) at (-.95,-.95);
				\foreach \u/\v in {a/c,a/d,b/d,c/d}
				\draw[tvedge] (\u)--(\v);
				\draw[tvheavy] (a)--(b)--(c);
				\tvvertex{a}{1}{5}{135}
				\tvvertex{b}{2}{7}{45}
				\tvvertex{c}{2}{6}{-45}
				\tvvertex{d}{1}{4}{-135}
			\end{tikzpicture}
		\end{minipage}\hfill
		\begin{minipage}[t]{.52\textwidth}
			\centering
			\textbf{(b)} $K_{3,3}$, $s=3$\par\medskip
			\begin{tikzpicture}
				\coordinate (a1) at (0,1.65);
				\coordinate (a2) at (1.6,1.65);
				\coordinate (a3) at (3.2,1.65);
				\coordinate (b1) at (0,0);
				\coordinate (b2) at (1.6,0);
				\coordinate (b3) at (3.2,0);
				\foreach \u/\v in {a1/b2,a1/b3,a2/b3,a3/b1,a3/b2,a3/b3}
				\draw[tvedge] (\u)--(\v);
				\draw[tvheavy] (a1)--(b1)--(a2)--(b2);
				\tvvertex{a1}{1}{6}{90}
				\tvvertex{a2}{1}{8}{90}
				\tvvertex{a3}{1}{4}{90}
				\tvvertex{b1}{2}{9}{-90}
				\tvvertex{b2}{2}{7}{-90}
				\tvvertex{b3}{2}{5}{-90}
			\end{tikzpicture}
		\end{minipage}
		\par\bigskip
		\textbf{(c)} $3K_4$, $s=4$\par\medskip
		\begin{tikzpicture}
			\begin{scope}[xshift=0cm]
				\coordinate (a) at (-.9,.9);
				\coordinate (b) at (.9,.9);
				\coordinate (c) at (.9,-.9);
				\coordinate (d) at (-.9,-.9);
				\foreach \u/\v in {a/b,a/c,a/d,b/c,b/d,c/d}
				\draw[tvedge] (\u)--(\v);
				\tvvertex{a}{1}{4}{135}
				\tvvertex{b}{2}{5}{45}
				\tvvertex{c}{3}{6}{-45}
				\tvvertex{d}{4}{7}{-135}
			\end{scope}
			\begin{scope}[xshift=4.1cm]
				\coordinate (a) at (-.9,.9);
				\coordinate (b) at (.9,.9);
				\coordinate (c) at (.9,-.9);
				\coordinate (d) at (-.9,-.9);
				\foreach \u/\v in {a/c,a/d,b/d}
				\draw[tvedge] (\u)--(\v);
				\draw[tvheavy] (a)--(b)--(c)--(d);
				\tvvertex{a}{2}{8}{135}
				\tvvertex{b}{1}{10}{45}
				\tvvertex{c}{2}{11}{-45}
				\tvvertex{d}{3}{9}{-135}
			\end{scope}
			\begin{scope}[xshift=8.2cm]
				\coordinate (a) at (-.9,.9);
				\coordinate (b) at (.9,.9);
				\coordinate (c) at (.9,-.9);
				\coordinate (d) at (-.9,-.9);
				\foreach \u/\v in {a/b,a/c,a/d,b/c,b/d,c/d}
				\draw[tvheavy] (\u)--(\v);
				\tvvertex{a}{1}{13}{135}
				\tvvertex{b}{2}{14}{45}
				\tvvertex{c}{3}{15}{-45}
				\tvvertex{d}{4}{16}{-135}
			\end{scope}
		\end{tikzpicture}
		\caption{Vertex irregular total labelings of the three exceptional cubic graphs. A number inside a vertex is its label, and the bracketed number beside it is its weight. Thick edges receive label $s$, with $s$ as specified in each panel; thin edges receive label $1$.}
		\label{fig:cubic-exceptions}
	\end{figure}
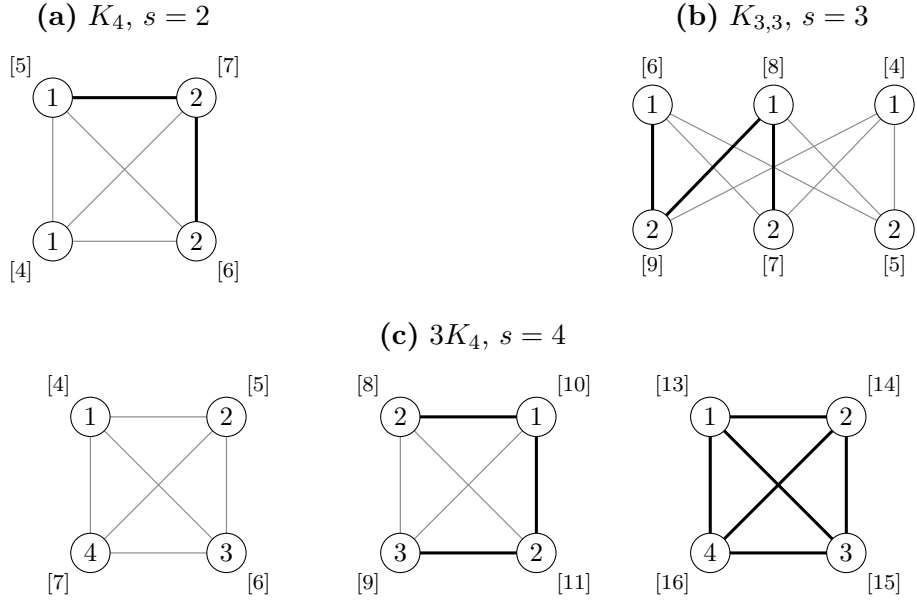
	
	\begin{proof}[Proof of Theorem~\ref{thm:cubic}]
		The lower bound is~\eqref{eq:regular-lower}.
		Since $n$ is even, we have $s=\ceil{\frac{n+3}{4}}=\ceil{\frac n4}+1$.

		Suppose first that $G\not\cong K_4,K_{3,3},3K_4$.  By Theorem~\ref{thm:LPS}, $G$ has a spanning subgraph $H$ with $n_i(H)\le\ceil{\frac n4}=s-1$ for each $i\in[0,3]$, and Lemma~\ref{lem:cubic-extension} gives the desired labeling.

		It remains to consider the three exceptional graphs, whose labelings are illustrated in Figure~\ref{fig:cubic-exceptions}.  Let $G=K_4$ with $V(G)=\{a,b,c,d\}$, so $s=2$.  Label the two adjacent edges $ab$ and $bc$ by $2$ and every other edge by $1$, and let $\lambda(a)=\lambda(d)=1$ and $\lambda(b)=\lambda(c)=2$.  Then $a,b,c,d$ have weights $5,7,6,4$, respectively.

		Let $G=K_{3,3}$, so $s=3$.  Let $H$ be a path on four vertices together with two isolated vertices.  Then $(n_0(H),n_1(H),n_2(H),n_3(H))=(2,2,2,0)$, and Lemma~\ref{lem:cubic-extension} applies.

		Finally, let $G=3K_4$, so $s=4$.  On the first copy of $K_4$ let $H$ be empty, on the second let $H$ be a spanning path, and on the third let $H$ be the whole copy.  Label the edges of $H$ by $4$ and all other edges by $1$.  On the first and the third copies, label the vertices by $1,2,3,4$.  On the second copy, label the two ends of the path by $2$ and $3$ and its two internal vertices by $1$ and $2$.  The three copies then have weight sets $[4,7]$, $[8,11]$, and $[13,16]$, respectively, so the labeling is vertex irregular.

		In each of the three cases, every edge receives label $1$ or $s$.  This completes the proof of Theorem~\ref{thm:cubic}.
	\end{proof}

	\section{Proof of Theorem~\ref{thm:four-main}}\label{sec:four}

	Throughout this section, $G$ is a $4$-regular graph on $n$ vertices, $s=\ceil{\frac{n+4}{5}}$, and $r=s-1=\ceil{\frac{n-1}{5}}$.  By the definition of $r$, we have
	\begin{equation}\label{eq:n-range}
		5r-3\le n\le 5r+1.
	\end{equation}
	The lower bound $\tvs(G)\ge s$ is~\eqref{eq:regular-lower}, so it suffices to construct a vertex irregular total $s$-labeling $\lambda$ of $G$ whose set of vertex weights is exactly $[5,n+4]$.

	We first outline the proof. Using a matching decomposition of Dalal, McDonald, and Shan~\cite{DMS}, Lemma~\ref{lem:smoothing} constructs a maximum matching $M$ and chooses the vertices it leaves uncovered so that each uncovered vertex can be \emph{smoothed}: we delete the vertex and pair its four neighbors by two \emph{virtual edges}. The pairings are chosen so that smoothing all uncovered vertices in $G-M$ gives a simple cubic graph $F$.

	Let $U$ be the set of uncovered vertices and $N=|V(F)|$. In the main case, Theorem~\ref{thm:LPS} gives a spanning subgraph $H$ of $F$ whose degree classes $V_0,V_1,V_2,V_3$ have sizes $a,b,a,b$, where $a=\floor{N/4}$ and $b=\ceil{N/4}$. We label the edges of $H$ by $r+1$ and the remaining edges of $F$ by $1$. Thus each vertex in $V_i$ receives a contribution of $3+ir$ from its incident edges of $F$.

	For $w\in U$, let $h(w)$ be the number of its two virtual edges that belong to $H$, and let $U_h=\{w\in U:h(w)=h\}$ for $h\in\{0,1,2\}$. When $w$ is restored, each virtual edge transfers its label to the two corresponding edges incident with $w$. These four edge labels sum to $4+2h(w)r$. We reserve three intervals of weights for $U_0,U_1,U_2$ and assign the remaining weights to $V_0,V_1,V_2,V_3$.

	More precisely, a vertex with target weight $5+j$ is assigned \emph{rank} $j$. We assign consecutive blocks of ranks to
	\[
	U_0,V_0,V_1,U_1,V_2,V_3,U_2
	\]
	in this order. For $v\in V_i$ with rank $R(v)$, let $t(v)=R(v)-ir$. Its vertex label and matching edge label must sum to $2+t(v)$. We show that $0\le t(v)\le2r$ for every $v\in V(F)$, and order the vertices within each $V_i$ so that $|t(u)-t(v)|\le r$ for every $uv\in M$. These inequalities allow both ends of each matching edge to attain their target weights using labels in $[1,r+1]$. The required ordering is given in Claim~\ref{clm:ordering}. We then assign the remaining labels and restore the vertices of $U$.

	This construction applies when $N\ge4r$ and $F\not\cong K_4,K_{3,3},3K_4$. If $N<4r$, then $G=kK_5\cup J_7$ for some $4$-regular graph $J_7$ on seven vertices. We treat this case in Section~\ref{subsec:small} and the three exceptional cubic graphs in Section~\ref{subsec:exceptional}.

	\subsection{The smoothing lemma}\label{subsec:smoothing}

	For a graph $J$, let $\alpha'(J)$ denote its maximum matching size and write $\defect(J)=|V(J)|-2\alpha'(J)$. A graph of odd order is \emph{factor-critical} if deleting any vertex leaves a graph with a perfect matching. Let $o(J)$ denote the number of odd components of $J$. A set $S\subseteq V(J)$ is a \emph{barrier} if $o(J-S)-|S|=\defect(J)$; a \emph{maximum barrier} is a barrier of maximum cardinality. The existence of a barrier follows from the Tutte--Berge formula; see Lov\'asz and Plummer~\cite{LP}. For $X\subseteq V(J)$, let $\partial X$ be the set of edges of $J$ with exactly one end in $X$, and write $\partial B=\partial V(B)$ for a subgraph $B$ of $J$. Finally, $K_5-e$ denotes the graph obtained from $K_5$ by deleting one edge.

	We use the following matching decomposition. Every maximum barrier is inclusionwise maximal, so the result applies to our choice of barrier below.

	\begin{externallemma}[Dalal, McDonald, and Shan~{\cite[Lemma~7]{DMS}}]
		Let $J$ be a graph and let $S$ be an inclusionwise maximal barrier of $J$. Define a bipartite multigraph $I$ with parts $S$ and the set of components of $J-S$, with one edge between $v\in S$ and a component $B$ for each edge of $J$ joining $v$ to $B$. Then the following hold.
		\begin{enumerate}[label=\textup{(\roman*)}]
			\item Every component of $J-S$ is factor-critical.
			\item $I$ has a matching that saturates $S$.
			\item If $S\ne\emptyset$ and $d=\max\{d_I(v):v\in S\}$, then the matching in \textup{(ii)} can be chosen to saturate every vertex of $V(I)\setminus S$ having degree at least $d$ in $I$.
		\end{enumerate}
	\end{externallemma}

	\begin{lemma}\label{lem:smoothing}
		Let $G$ be a $4$-regular graph on $n$ vertices. Then $G$ has a maximum matching $M$ such that, with $U$ denoting the set of vertices not covered by $M$, the following hold.
		\begin{enumerate}[label=\textup{(\roman*)},ref=\textup{(\roman*)}]
			\item\label{it:smooth-q} $|U|\le n/5$.
			\item\label{it:smooth-def} If $J$ is a component of $G$ that has no perfect matching and is not factor-critical, then
			\begin{equation}\label{eq:11def}
				|V(J)|\ge 11\defect(J).
			\end{equation}
			\item\label{it:smooth-F} For every $w\in U$, the four neighbors of $w$ can be partitioned into two pairs $\{x,y\}$ and $\{x',y'\}$ such that the graph $F$ obtained from $G-U-M$ by adding, for every $w\in U$, the two \emph{virtual edges} $xy$ and $x'y'$ is a simple cubic graph.
		\end{enumerate}
	\end{lemma}

	\begin{proof}
		Let $S$ be a maximum barrier of $G$. We first observe that every odd component $B$ of $G-S$ with exactly two edges to $S$ has at least five vertices. Indeed,
		\[
		|\partial B|=4|V(B)|-2|E(B)|,
		\]
		so a singleton has four boundary edges, while a set of three vertices has at least six boundary edges because $G$ is simple. An odd component with no edges to $S$ is itself a $4$-regular component of $G$ and also has at least five vertices. The same identity shows that every component of $G-S$ has an even number of boundary edges.

		Apply the matching decomposition above with $J=G$, and let $I$ be its bipartite multigraph. Every component of $G-S$ is factor-critical. If $S\ne\emptyset$, each vertex of $S$ has degree at most four in $I$, so there is a matching $L$ that saturates $S$ and every vertex corresponding to a component with at least four boundary edges. If $S=\emptyset$, take $L=\emptyset$; in this case every component has boundary zero.

		Let $c_0$ and $c_2$ be the numbers of components of $G-S$ with zero and two boundary edges, respectively. Since all components of $G-S$ have odd order, exactly $o(G-S)-|S|=\defect(G)$ of them are not covered by $L$. Each of these has zero or two boundary edges. By the preceding observation, the $c_0+c_2$ components counted here are vertex-disjoint and each has at least five vertices. Hence
		\[
		\defect(G)\le c_0+c_2\le \frac{n}{5}.
		\]

		Lift each edge of $L$ to its corresponding edge of $G$. If such an edge meets a component $B$ of $G-S$ at $v$, add a perfect matching of $B-v$. In each component not covered by $L$, choose one vertex to leave uncovered and match the remaining vertices internally. These choices are possible because every component of $G-S$ is factor-critical. The resulting matching $M$ misses exactly $\defect(G)$ vertices and is therefore maximum. Thus its uncovered set $U$ satisfies $|U|=\defect(G)\le n/5$, proving~\ref{it:smooth-q}. We specify the uncovered vertices and the internal matchings in the unmatched components below.

		To prove~\ref{it:smooth-def}, let $J$ be a component of $G$ that has no perfect matching and is not factor-critical, and put $S_J=S\cap V(J)$. The Tutte--Berge formula and additivity of deficiency over components imply that $S_J$ is a barrier of $J$. Moreover, $S_J\ne\emptyset$, since otherwise $J$ would be a component of $G-S$ and hence factor-critical. Every component of $J-S_J$ therefore has a positive even number of boundary edges. Let $\ell$ and $m$ be the numbers of these components with exactly two and at least four boundary edges, respectively, and put
		\[
		q_J=\defect(J)=\ell+m-|S_J|.
		\]
		Counting edges from $J-S_J$ to $S_J$ gives
		\[
		2\ell+4m\le 4|S_J|=4(\ell+m-q_J),
		\]
		and hence $\ell\ge 2q_J$.

		Each of the $\ell$ components with two boundary edges has at least five vertices by the initial observation. Consequently,
		\[
		\begin{aligned}
			|V(J)|
			&\ge |S_J|+5\ell+m\\
			&=6\ell+2m-q_J\\
			&\ge 11q_J.
		\end{aligned}
		\]
		This proves~\ref{it:smooth-def}.

		It remains to specify the uncovered vertices and their virtual edges. Let $B$ be a component of $G-S$ not covered by $L$. Then $B$ is factor-critical and $|\partial B|\in\{0,2\}$. If $|\partial B|=0$, then $B$ is a component of $G$.
		In choosing the uncovered vertex $w$ of $B$ and the two virtual edges at $w$, we shall ensure the following \emph{auxiliary condition}: each virtual edge at $w$ has at least one end in $B-w$. This condition will be used below to show that $F$ is simple.

		\medskip
		\noindent\textbf{Case 1.} \emph{Some vertex $w$ of $B$ lies in no copy of $K_4$ in $G$.}

		\medskip
		Leave $w$ uncovered and choose any perfect matching of $B-w$. We choose the virtual edges at these vertices after all internal matchings have been fixed.

		\medskip
		\noindent\textbf{Case 2.} \emph{Every vertex of $B$ lies in a copy of $K_4$ in $G$.}

		\medskip
		Every copy of $K_4$ meeting $B$ lies entirely in $B$. This is immediate when $B$ is a component of $G$. Otherwise, $|\partial B|=2$, whereas a copy of $K_4$ with exactly $k\in\{1,2,3\}$ vertices in $B$ would contribute $k(4-k)\ge3$ edges to $\partial B$.

		Two distinct copies of $K_4$ cannot share exactly one or two vertices, since a shared vertex would then have degree at least five. If they share three vertices, their union induces a $K_5-e$ or a $K_5$. In the latter case, this $K_5$ is a component of $G$, so $B=K_5$. In the former case, the three shared vertices already have degree four within the union, and each of the other two vertices has only one neighbor outside it. Thus no further copy of $K_4$ meets the union. It follows that, unless $B=K_5$, the vertices of $B$ partition into clusters, each inducing a $K_4$ or a $K_5-e$.

		If $B=K_5$, leave any vertex $w$ uncovered, choose a perfect matching of $B-w$, and prescribe its two edges as the virtual edges at $w$.

		Otherwise, since $|V(B)|$ is odd, some cluster induces a $K_5-e$. Write
		\[
		Q=G[\{p,q,x,y,w\}]=K_5-pq.
		\]
		Only $p$ and $q$ have neighbors outside $Q$. Since $B$ is factor-critical, $B-p$ has a perfect matching. This matching uses no edge between $Q-p$ and $B-Q$: the only possible such edge is incident with $q$, and using it would leave the three vertices $x,y,w$ to be matched among themselves. Hence $B-Q$ has a perfect matching. Choose one, add $px$ and $qy$, leave $w$ uncovered, and prescribe $px$ and $qy$ as the virtual edges at $w$.

		\medskip
		Having made these choices in every component $B$ of $G-S$ not covered by $L$, let $M$ be the resulting maximum matching, let $U$ be its uncovered set, and put $R=G-U-M$. The virtual edges prescribed in Case~2 belong to $M$ and are therefore absent from $R$.

		Now consider a vertex $w$ chosen in Case~1. Since $M$ is maximum, $U$ is independent, so every neighbor of $w$ is covered by $M$. Each such neighbor loses its matching edge and its edge to $w$ when passing from $G$ to $R$. Thus
		\[
		\Delta(R[N_G(w)])\le2,
		\]
		and the complement of $R[N_G(w)]$ has minimum degree at least one. A graph on four vertices with minimum degree at least one and no perfect matching is a $K_{1,3}$. If this complement were a $K_{1,3}$, then $R[N_G(w)]$ would contain a triangle, which together with $w$ would form a $K_4$ in $G$. This contradicts the choice of $w$. Hence the complement has a perfect matching, and we take the virtual edges at $w$ to be the two edges of a perfect matching of the complement, chosen as follows.

		Since $|\partial B|\le2$, the vertex $w$ has at most two neighbors outside $B$. If it has at most one, we take any perfect matching of the complement: every pairing of its four neighbors gives two virtual edges, each with at least one end in $B-w$. In particular, when it has exactly one outside neighbor, that neighbor is paired with an inside neighbor, and the remaining pair lies entirely in $B-w$.

		If $w$ has two neighbors outside $B$, then both boundary edges are incident with $w$. Each outside neighbor is therefore nonadjacent in $G$ to both inside neighbors. In this situation, we pair each outside neighbor with an inside neighbor. These pairs are nonedges of $G$ and hence of $R$, so they form a perfect matching of the complement, and both have an end in $B-w$. Thus the auxiliary condition holds in Case~1. It also holds in Case~2, since both ends of each prescribed virtual edge lie in $B-w$.

		Let $F$ be obtained from $R$ by adding all the virtual edges. The two virtual edges at any fixed vertex of $U$ have distinct ends and are absent from $R$. To rule out repetitions among virtual edges, let $w_1,w_2$ be distinct uncovered vertices, lying in distinct components $B_1,B_2$ of $G-S$. By the auxiliary condition, each virtual edge at $w_1$ has an end in $B_1-w_1$, whereas both ends of every virtual edge at $w_2$ lie in $V(B_2)\cup S$. Since these vertex sets are disjoint, the edges cannot coincide. Hence $F$ is simple.

		Finally, every vertex $v$ of $F$ is covered by $M$. In passing from $G$ to $F$, it loses its matching edge and one edge for each neighbor in $U$, and gains one virtual edge for each such neighbor. Therefore
		\[
		d_F(v)=4-1-|N_G(v)\cap U|+|N_G(v)\cap U|=3.
		\]
		Hence $F$ is cubic, proving~\ref{it:smooth-F}.
	\end{proof}

	The matching edges will be labeled together with their ends. Restoring a vertex of $U$ transfers the labels of its virtual edges to its four incident edges and preserves the edge-label sum at every vertex covered by $M$. The bound in Lemma~\ref{lem:smoothing}\ref{it:smooth-q} ensures that the reserved ranks lie in the required ranges. We use Lemma~\ref{lem:smoothing}\ref{it:smooth-def} in Sections~\ref{subsec:small} and~\ref{subsec:exceptional}.

	For the rest of this section, let $M$, $U$, and $F$ be as in Lemma~\ref{lem:smoothing}. Let $q=|U|$ and $N=n-q$. Then $N$ is even and $q\le n/5$. A virtual edge of $F$ may have the same ends as an edge of $M$. We regard these as distinct edges and assign their labels independently; only the virtual edge is replaced when the corresponding vertex of $U$ is restored. Let $a=\floor{N/4}$ and $b=\ceil{N/4}$. Since $N$ is even,
	\begin{equation}\label{eq:Nab}
		N=2a+2b\qquad\text{and}\qquad b\in\{a,a+1\}.
	\end{equation}

	\subsection[The main case]{The main case: $N\ge4r$ and $F\not\cong K_4,K_{3,3},3K_4$}\label{subsec:main}

	Assume that
	\begin{equation}\label{eq:main-range}
		N\ge4r.
	\end{equation}
	Then $a\ge r$.  We also have $q\le r$, as $q\le\floor{\frac n5}\le\ceil{\frac{n-1}5}=r$.  Since $n+q$ is even and $n+q\le5r+1+r=6r+1$ by~\eqref{eq:n-range}, we obtain
	\begin{equation}\label{eq:num1}
		a+b+q=\frac{n+q}{2}\le3r.
	\end{equation}
	Moreover, $2a+b=\floor{3N/4}$ by~\eqref{eq:Nab}, and $3N+4q=3n+q\le 16r+3$ by~\eqref{eq:n-range} and $q\le r$, so
	\begin{equation}\label{eq:num2}
		2a+b+q
		=\floor{\frac{3n+q}{4}}
		\le4r.
	\end{equation}

	Assume in addition that $F\not\cong K_4,K_{3,3},3K_4$.  By Theorem~\ref{thm:LPS}, $F$ has a spanning subgraph $H$ whose four degree classes have sizes $a$ or $b$.  When $N\equiv2\pmod4$, the number of vertices of $H$ of odd degree is even, so the two odd-degree classes have the same size. Replacing $H$ by $F-E(H)$ if necessary, we may assume that
	\begin{equation}\label{eq:fourclasses}
		(|V_0|,|V_1|,|V_2|,|V_3|)=(a,b,a,b),
	\end{equation}
	where $V_i=V_i(H)=\{v\in V(F):d_H(v)=i\}$.

	We label the edges of $F$ by
	\begin{equation}\label{eq:Fedge-label}
		\lambda(e)=
		\begin{cases}
			r+1,&\text{if } e\in E(H),\\
			1,&\text{if } e\notin E(H).
		\end{cases}
	\end{equation}
	Thus a vertex in $V_i$ receives a contribution of $3+ir$ from its three incident edges of $F$.

	Recall that each $w\in U$ was replaced by two virtual edges.  Let $h(w)\in\{0,1,2\}$ be the number of these two virtual edges that belong to $H$, and for $h\in\{0,1,2\}$ let $U_h=\{w\in U: h(w)=h\}$.  Let $\alpha=|U_0|$, $\beta=|U_1|$, and $\gamma=|U_2|$, so that $\alpha+\beta+\gamma=q$.  By $n\ge5$, \eqref{eq:main-range}, \eqref{eq:Nab}, \eqref{eq:num1}, \eqref{eq:num2}, and $\beta\le q$, we have
	\begin{equation}\label{eq:ord-hyp}
		r\ge1,\qquad r\le a\le b\le a+1,\qquad a+b+\beta\le3r,\qquad 2a+b+\beta\le4r.
	\end{equation}

	\medskip
	\noindent\textbf{Target weights.}
	The vertex set of $G$ is now partitioned into the seven sets $U_0,V_0,V_1,U_1,V_2,V_3,U_2$, and we aim for the following weights:
	\begin{equation}\label{eq:targets}
		\begin{array}{c@{\qquad}c@{\qquad}l}
			\text{set} & \text{size} & \text{target weights}\\ \hline
			U_0 & \alpha & [\,5,\ 4+\alpha\,]\\
			V_0 & a & [\,5+\alpha,\ 4+\alpha+a\,]\\
			V_1 & b & [\,5+\alpha+a,\ 4+\alpha+N/2\,]\\
			U_1 & \beta & [\,5+\alpha+N/2,\ 4+\alpha+\beta+N/2\,]\\
			V_2 & a & [\,5+\alpha+\beta+N/2,\ 4+\alpha+\beta+N/2+a\,]\\
			V_3 & b & [\,5+\alpha+\beta+N/2+a,\ 4+\alpha+\beta+N\,]\\
			U_2 & \gamma & [\,5+\alpha+\beta+N,\ n+4\,]
		\end{array}
	\end{equation}
	Since $a+b=N/2$ and $N+\alpha+\beta+\gamma=n$, the nonempty intervals in~\eqref{eq:targets} partition $[5,n+4]$, and each interval has the size of its corresponding set. For $w\in U_h$, the four incident edge labels will sum to $4+2hr$, so its attainable weights lie in $[5+2hr,5+(2h+1)r]$. For $v\in V_i$, the matching edge label and vertex label together contribute between $2$ and $2r+2$, so its attainable weights lie in $[5+ir,5+(i+2)r]$. We verify below that the target intervals lie in these ranges. We must also order the vertices within each $V_i$ so that the two ends of every matching edge can attain their target weights using a common edge label.

	\medskip
	\noindent\textbf{Ranks.}
	A target weight of $5+j$ corresponds to rank $j$.  In terms of ranks, the targets for $U_0,U_1,U_2$ in~\eqref{eq:targets} are the three intervals
	\[
	J_0=[0,\alpha-1],\qquad
	J_1=[\alpha+N/2,\ \alpha+N/2+\beta-1],\qquad
	J_2=[N+\alpha+\beta,\ n-1],
	\]
	some of which may be empty.  From $q\le r$, \eqref{eq:main-range}, and~\eqref{eq:num1}, we have
	\begin{equation}\label{eq:reserved-ranges}
		J_0\subseteq[0,r],\qquad
		J_1\subseteq[2r,3r],\qquad
		J_2\subseteq[4r,5r].
	\end{equation}
	Indeed, $N/2\ge2r$ and $N/2+q\le3r$, while $n-1\le5r$ by~\eqref{eq:n-range}.

	Let $\rho_0=0$, $\rho_1=a$, $\rho_2=a+b$, and $\rho_3=2a+b$.
	Thus $\rho_i=\sum_{j=0}^{i-1}|V_j|$ is the number of vertices in the
	classes preceding $V_i$ when the four classes are listed as
	$V_0,V_1,V_2,V_3$.
	Fix a linear order $v_{i,1},\ldots,v_{i,|V_i|}$ on each $V_i$.
	For $v=v_{i,k}$, define $\rho(v)=k-1$; equivalently, $\rho(v)$ is
	the number of vertices of $V_i$ preceding $v$ in this order.
	In particular, the first and last vertices of $V_i$ have positions
	$0$ and $|V_i|-1$, respectively.

	We assign ranks by listing the seven blocks in the order
	$U_0,V_0,V_1,U_1,V_2,V_3,U_2$, starting with rank $0$.
	Before a vertex $v\in V_i$, there are all $\alpha$ vertices of $U_0$,
	the $\rho_i$ vertices in the preceding $V$-classes, and the
	$\rho(v)$ vertices preceding $v$ within $V_i$.
	The $\beta$ vertices of $U_1$ also precede $v$ exactly when $i\ge2$.
	Consequently, the rank of $v$ is
	\begin{equation}\label{eq:rank}
		R(v)=\alpha+\rho_i+\rho(v)+\beta\mathbf1_{\{i\ge2\}},
	\end{equation}
	where $\mathbf1_{\{i\ge2\}}$ is $1$ when $i\ge2$ and $0$ otherwise.
	The term involving $\beta$ therefore accounts for the interval $J_1$
	reserved for $U_1$ between $V_1$ and $V_2$.
	The vertices of $U_2$ occur after all four $V$-classes and do not
	contribute to $R(v)$.

	The ranks assigned to $V_i$ are precisely its target weights
	in~\eqref{eq:targets}, each decreased by $5$.
	Thus the ranks assigned to the vertices of $F$ are exactly
	$[0,n-1]\setminus(J_0\cup J_1\cup J_2)$.

	For $v\in V_i$, let
	\begin{equation}\label{eq:t-def}
		t(v)=R(v)-ir.
	\end{equation}
	We claim that, regardless of the orders within the classes,
	\begin{equation}\label{eq:t-range}
		0\le t(v)\le2r\qquad\text{for every }v\in V(F).
	\end{equation}
	Since $a,b\ge r$, we have $\rho_i\ge ir$ for each $i\in[0,3]$;
	in particular, $\rho_3=2a+b\ge3r$. All other terms in
	$R(v)-\rho_i$ are nonnegative, so $t(v)\ge0$.

	For the upper bound, the maximum values of $t$ on
	$V_0,V_1,V_2,V_3$, respectively, are
	\[
	\alpha+a-1,\qquad
	\alpha+a+b-1-r,\qquad
	\alpha+\beta+2a+b-1-2r,\qquad
	\alpha+\beta+N-1-3r.
	\]
	The successive differences are $b-r$, $a+\beta-r$, and $b-r$,
	all of which are nonnegative. Hence, by~\eqref{eq:n-range},
	\[
	t(v)\le\alpha+\beta+N-1-3r
	=n-\gamma-1-3r\le2r
	\]
	for every $v\in V(F)$. This proves~\eqref{eq:t-range}.

	\medskip
	\noindent\textbf{Ordering the classes.}
	We now choose the orders within $V_0,\ldots,V_3$ to control the
	differences of $t$ along the edges of $M$.

	Let $uv\in M$ with $u\in V_i$ and $v\in V_j$.
	By~\eqref{eq:Fedge-label}, the three edges of $F$ incident with $u$
	have labels summing to $3+ir$.
	The two contributions still to be assigned at $u$ are its vertex label
	$\lambda(u)$ and the label $\lambda(uv)$ of its unique incident
	matching edge.
	Thus attaining the target weight $5+R(u)$ requires
	$(3+ir)+\lambda(uv)+\lambda(u)=5+R(u)$.
	Subtracting $3+ir$ gives
	$\lambda(uv)+\lambda(u)=2+R(u)-ir=2+t(u)$,
	where the last equality uses~\eqref{eq:t-def}.
	Applying the same calculation at $v$, whose incident edges of $F$
	have labels summing to $3+jr$, gives the two requirements

	\begin{equation}\label{eq:need}
		\lambda(uv)+\lambda(u)=2+t(u)
		\qquad\text{and}\qquad
		\lambda(uv)+\lambda(v)=2+t(v).
	\end{equation}
	Subtracting these equations gives
	$\lambda(v)-\lambda(u)=t(v)-t(u)$. Since the vertex labels lie in
	$[1,r+1]$, it is necessary that $|t(u)-t(v)|\le r$.

	Together with~\eqref{eq:t-range}, this condition is also sufficient.
	Indeed, writing $\lambda(uv)=1+z$, the three labels lie in $[1,r+1]$
	precisely when
	\begin{equation}\label{eq:z-window}
		\max\{0,t(u)-r,t(v)-r\}
		\le z\le
		\min\{r,t(u),t(v)\}.
	\end{equation}
	By~\eqref{eq:t-range}, this interval is nonempty if and only if
	$|t(u)-t(v)|\le r$. It therefore remains to prove the following claim.

	\begin{claim}\label{clm:ordering}
		The classes $V_0,V_1,V_2,V_3$ can be ordered so that
		\begin{equation}\label{eq:matching-diff}
			|t(u)-t(v)|\le r\qquad\text{for every }uv\in M.
		\end{equation}
	\end{claim}

	We postpone the proof of Claim~\ref{clm:ordering} to
	Section~\ref{subsec:claimproof} and first complete the construction.
	Fix the orders given by the claim, and define $\rho(v)$, $R(v)$,
	and $t(v)$ with respect to these orders.
	Then~\eqref{eq:t-range} and~\eqref{eq:matching-diff} both hold.

	\medskip
	\noindent\textbf{Labels and restoration.}
	For each $uv\in M$, let
	$z_{uv}=\max\{0,t(u)-r,t(v)-r\}$.
	By~\eqref{eq:t-range} and~\eqref{eq:matching-diff},
	\[
	0\le z_{uv}\le r,\qquad
	0\le t(u)-z_{uv}\le r,\qquad
	0\le t(v)-z_{uv}\le r.
	\]
	Define
	\[
	\lambda(uv)=1+z_{uv},\qquad
	\lambda(u)=1+t(u)-z_{uv},\qquad
	\lambda(v)=1+t(v)-z_{uv}.
	\]
	All these labels lie in $[1,r+1]=[1,s]$.
	For $v\in V_i$ with matching partner $u$, we have
	\begin{equation}\label{eq:matched-weight}
		\wt_{\lambda}(v)
		=(3+ir)+(1+z_{uv})+(1+t(v)-z_{uv})
		=5+R(v),
	\end{equation}
	where the weight is computed using the three incident edges of $F$
	and the matching edge. A virtual edge and a matching edge with the
	same ends are treated as distinct edges.
	Thus the vertices of $V_i$ receive exactly their target weights
	in~\eqref{eq:targets}.

	It remains to restore the vertices of $U$.
	For each virtual edge $xy$ at $w\in U$, delete the virtual edge
	and assign its label to both edges $xw$ and $wy$ of $G$.
	This preserves the weight of every vertex covered by $M$,
	so~\eqref{eq:matched-weight} holds in $G$.

	For each $h\in[0, 2]$, assign the ranks in $J_h$ bijectively to
	the vertices of $U_h$. If $w\in U_h$ is assigned rank $j$, let
	$\lambda(w)=1+j-2hr$. By~\eqref{eq:reserved-ranges}, this label
	lies in $[1,r+1]$. The four incident edge labels at $w$ sum to
	$4+2hr$, so
	\[
	\wt_{\lambda}(w)=4+2hr+\lambda(w)=5+j.
	\]
	Together with~\eqref{eq:matched-weight}, this shows that the vertex
	weights are precisely $[5,n+4]$.
	Thus $\lambda$ is a vertex irregular total $s$-labeling of $G$,
	proving Theorem~\ref{thm:four-main} in the main case.

	\subsection[The case N < 4r]{The case $N<4r$}\label{subsec:small}

	Suppose that $N<4r$.
	By~\eqref{eq:n-range}, we have $n\ge5r-3$ and
	$q\le\floor{n/5}\le r$.
	Hence $N=n-q\ge4r-3$.
	Since $N$ is even and $N<4r$, it follows that $N=4r-2$.
	Now $n=N+q\ge5r-3$ gives $q\ge r-1$.
	On the other hand, $5q\le n=N+q=4r-2+q$ gives
	$4q\le4r-2$, so the integrality of $q$ implies $q\le r-1$.
	Consequently, $q=r-1$ and $n=4r-2+q=5q+2$.

	For each component $J$ of $G$, let $q_J$ be the number of vertices of $J$ not covered by $M$; as $M\cap E(J)$ is a maximum matching of $J$, we have $q_J=\defect(J)$.  A component with a perfect matching has $q_J=0$ and at least six vertices.  A component without a perfect matching that is not factor-critical satisfies $|V(J)|-5q_J\ge6q_J$ by~\eqref{eq:11def}.  A factor-critical component has $q_J=1$, and $|V(J)|-5$ is a nonnegative even integer.  Since $\sum_J(|V(J)|-5q_J)=n-5q=2$, every component of $G$ is factor-critical; all but one of them have five vertices, and the remaining one has seven vertices.  Hence
	\begin{equation}\label{eq:small-structure}
		G=kK_5\cup J_7
	\end{equation}
	for some integer $k\ge0$ and some $4$-regular graph $J_7$ on seven vertices.

	The labelings of $K_5$ and $J_7$ used below are illustrated in Figure~\ref{fig:small-labelings}.

	We first label $K_5$.  Let $H$ be a spanning subgraph of $K_5$ consisting of a triangle with one pendant edge attached, together with an isolated vertex.  Label the four edges of $H$ by $2$ and all other edges by $1$.  Label the isolated vertex, the pendant vertex, and one vertex of degree $2$ in $H$ by $1$, and the other vertex of degree $2$ in $H$ and the vertex of degree $3$ in $H$ by $2$.  The weights are $5,6,7,8,9$.

	We next label $J_7$.  Every $4$-regular graph on seven vertices contains a triangle: otherwise the four neighbors of a vertex would be independent, while each of them would need three further neighbors among the two remaining vertices.  Let $H$ consist of the three edges of a triangle in $J_7$ together with one edge leaving the triangle.  Label the edges of $H$ by $3$ and all other edges by $1$.  Label the three vertices of degree $0$ in $H$ by $1,2,3$, the vertex of degree $1$ by $2$, the two vertices of degree $2$ by $1,2$, and the vertex of degree $3$ by $1$.  The weights are $5,6,\ldots,11$.

	\begin{figure}[htbp]
		\centering
		\begin{minipage}[t]{.45\textwidth}
			\centering
			\textbf{(a)} $K_5$\par\medskip
			\begin{tikzpicture}
				\coordinate (a) at (0,1.4);
				\coordinate (b) at (-1.35,.3);
				\coordinate (c) at (1.35,.3);
				\coordinate (d) at (.9,-1.3);
				\coordinate (e) at (-.9,-1.3);
				\foreach \u/\v in {a/e,b/d,b/e,c/d,c/e,d/e}
				\draw[tvedge] (\u)--(\v);
				\draw[tvheavy] (a)--(b)--(c)--(a)--(d);
				\tvvertex{a}{2}{9}{90}
				\tvvertex{b}{1}{7}{180}
				\tvvertex{c}{2}{8}{0}
				\tvvertex{d}{1}{6}{-45}
				\tvvertex{e}{1}{5}{-135}
			\end{tikzpicture}
			\par\smallskip
			{\small Thick edges: label $2$.\par Thin edges: label $1$.}
		\end{minipage}\hfill
		\begin{minipage}[t]{.52\textwidth}
			\centering
			\textbf{(b)} The spanning subgraph $H$ of $J_7$\par\medskip
			\begin{tikzpicture}
				\coordinate (a) at (0,.6);
				\coordinate (b) at (-1.25,-.65);
				\coordinate (c) at (1.25,-.65);
				\coordinate (d) at (0,2.05);
				\coordinate (e) at (-1.4,-2.0);
				\coordinate (f) at (0,-2.0);
				\coordinate (g) at (1.4,-2.0);
				\draw[tvheavy] (a)--(b)--(c)--(a)--(d);
				\tvvertex{a}{1}{11}{180}
				\tvvertex{b}{1}{9}{180}
				\tvvertex{c}{2}{10}{0}
				\tvvertex{d}{2}{8}{90}
				\tvvertex{e}{1}{5}{-90}
				\tvvertex{f}{2}{6}{-90}
				\tvvertex{g}{3}{7}{-90}
			\end{tikzpicture}
			\par\smallskip
			{\small Shown edges: label $3$.\par Omitted edges of $J_7$: label $1$.}
		\end{minipage}
		\caption{The labelings used when $N<4r$. A number inside a vertex is its label, and the bracketed number beside it is its weight in the ambient $4$-regular graph. Panel~(a) shows all of $K_5$. Panel~(b) shows only the four edges of $H$, not all edges of $J_7$; its three isolated vertices are isolated in $H$, not in $J_7$. Every edge of $J_7-E(H)$ receives label $1$.}
		\label{fig:small-labelings}
	\end{figure}
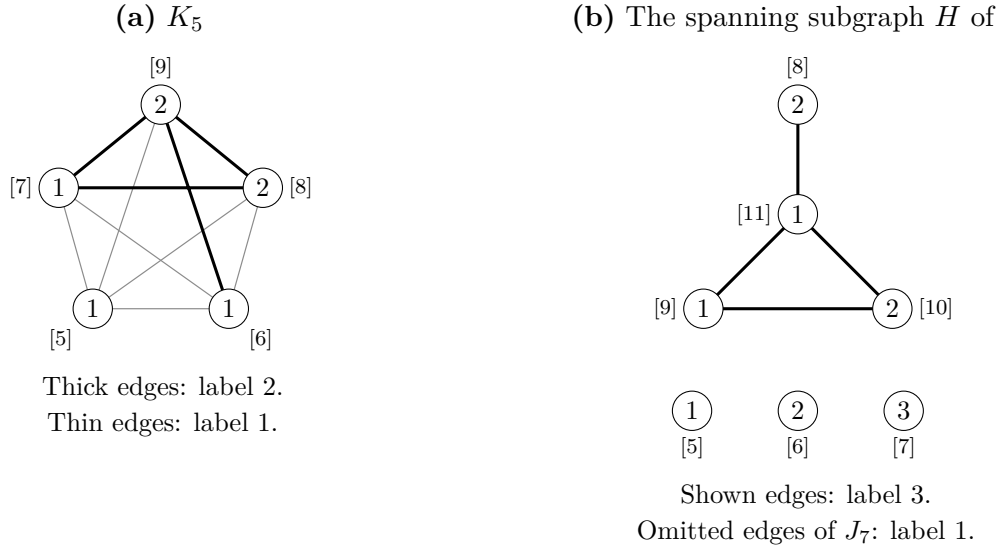

	Now consider $G$ as in~\eqref{eq:small-structure}, and index its copies of $K_5$ by $[0,k-1]$.  For $j\in[0,k-1]$, label copy $j$ by the labeling above with $j$ added to every vertex and edge label; this shifts each weight by $5j$.  Label $J_7$ by the labeling above with $k$ added to every label.  The weights of the resulting labeling of $G$ are $5,6,\ldots,n+4$, and the largest label is $k+3=\ceil{\frac{n+4}{5}}$.  This proves Theorem~\ref{thm:four-main} when $N<4r$.

	\subsection{The exceptional cubic graphs}\label{subsec:exceptional}

	It remains to consider the case $N\ge4r$ and $F\cong K_4,K_{3,3},3K_4$.  Since $q\le n/5$ and $n=N+q$, we have $q\le N/4$.

	Suppose first that $F\cong K_4$.  Then $q\le1$, and the only $4$-regular graph $G$ on $N+q\le5$ vertices is $K_5$, which was labeled in Section~\ref{subsec:small}.

	Suppose next that $F\cong K_{3,3}$.  Then $q\le1$.  If $q=1$, then $n=7$, which falls under the case $N<4r$ treated in Section~\ref{subsec:small}.  If $q=0$, then $F=G-M$, so $M$ would be a perfect matching in the complement of $K_{3,3}$; but that complement is $2K_3$, which has no perfect matching.

	Finally, suppose that $F\cong 3K_4$, so $N=12$ and $q\le3$.  Let $Q_0,Q_1,Q_2$ be the three copies of $K_4$ in $F$.

	Assume first that $q=0$.  Label every edge of $M$ by $1$, and label the copies as follows.  On $Q_0$, label every edge by $1$ and the vertices by $1,2,3,4$; the weights are $5,6,7,8$.  On $Q_1$, label the edges of a spanning path by $4$ and the other three edges by $1$; label the two ends of the path by $2,3$ and its two internal vertices by $1,2$; the weights are $9,10,11,12$.  On $Q_2$, label every edge by $4$ except one edge, which receives label $3$; label the two ends of this edge by $1,2$ and the other two vertices by $2,3$; the weights are $13,14,15,16$.

	Assume next that $q=1$, and let $w$ be the vertex of $U$.  We use the labeling for $q=0$, but arrange that both virtual edges at $w$ receive label $4$.  If the two virtual edges lie in the same copy of $K_4$, we take that copy to be $Q_2$ and choose the edge of label $3$ in $Q_2$ to be a different edge.  If they lie in different copies, we take these copies to be $Q_1$ and $Q_2$, choosing the spanning path in $Q_1$ to contain its virtual edge and the edge of label $3$ in $Q_2$ to be different from its virtual edge.  Restoring $w$ and labeling it by $1$ gives $w$ weight $17$.

	Suppose that $q=2$, so $n=14$.
	For each component $J$ of $G$, let $q_J=|V(J)\cap U|$ be the
	number of vertices of $J$ not covered by $M$, as in Section~\ref{subsec:small}.
	If $J$ has no perfect matching and is not factor-critical, then
	Lemma~\ref{lem:smoothing}\ref{it:smooth-def} gives
	$|V(J)|\ge11q_J$.
	Such a component containing both vertices of $U$ would therefore
	have at least $22$ vertices.
	If it contained exactly one vertex of $U$, it would have at least
	$11$ vertices, leaving at most three vertices for the component
	containing the other vertex of $U$.
	Both possibilities are impossible, since $n=14$ and every
	$4$-regular component has at least five vertices.

	Thus each vertex of $U$ lies in a factor-critical component.
	A maximum matching of a factor-critical graph leaves exactly one
	vertex uncovered.
	Since the restriction of $M$ to each component of $G$ is a maximum
	matching of that component, each factor-critical component contains
	exactly one vertex of $U$.
	Hence the two vertices of $U$ lie in two distinct factor-critical
	components of $G$.
	These are the only components: three components would require
	at least $15$ vertices.

	Smoothing and deleting matching edges do not create edges between
	distinct components of $G$.
	Since $F\cong3K_4$, each component of $G$ therefore contributes a
	multiple of four vertices to $F$.
	Each of the two factor-critical components loses exactly one vertex
	when passing to $F$, so their orders are congruent to $1$ modulo $4$.
	Their orders are at least five and sum to $14$; consequently,
	$G=K_5\cup J_9$ for some $4$-regular graph $J_9$ on nine vertices.

	The main case of Section~\ref{subsec:main} applies to $J_9$:
	its smoothed cubic graph has eight vertices, and $2K_4$ is not an
	exception in Theorem~\ref{thm:LPS}.
	We obtain a labeling of $J_9$ with labels at most $3$ and weights
	$5,\ldots,13$.
	Adding $1$ to every label on $J_9$ shifts these weights to
	$10,\ldots,18$, and we label $K_5$ as in
	Section~\ref{subsec:small} to obtain the weights $5,\ldots,9$.
	All labels are at most $4$.

	Finally, suppose that $q=3$, so $n=15=5q$.  With $q_J$ as above, the argument of Section~\ref{subsec:small} shows that every term of $\sum_J(|V(J)|-5q_J)=n-5q=0$ is nonnegative, and that a term vanishes only when $J$ is a factor-critical graph on five vertices. Hence every component of $G$ is a factor-critical graph on five vertices, so $G=3K_5$.  We label the three copies of $K_5$ as in Section~\ref{subsec:small}, adding $0,1,2$, respectively, to every label.  The weights are $5,\ldots,19$, and all labels are at most $4$.

	In all cases we have constructed a vertex irregular total $s$-labeling of $G$ with weight set $[5,n+4]$.  Together with the lower bound~\eqref{eq:regular-lower}, this shows $\tvs(G)=s$.  This proves Theorem~\ref{thm:four-main}. \qed

	\subsection{Proof of Claim~\ref{clm:ordering}}\label{subsec:claimproof}

	Recall that $M$ is a perfect matching on $V(F)=V_0\cup V_1\cup V_2\cup V_3$, that $|V_0|=|V_2|=a$ and $|V_1|=|V_3|=b$ by~\eqref{eq:fourclasses}, and that $r,a,b,\beta$ satisfy~\eqref{eq:ord-hyp}.

	\begin{claimproof}[Proof of Claim~\ref{clm:ordering}]

		We must choose a linear order on each of $V_0,V_1,V_2,V_3$ so that,
		with $R(v)$ and $t(v)$ defined by~\eqref{eq:rank}
		and~\eqref{eq:t-def}, we have $|t(u)-t(v)|\le r$ for every
		$uv\in M$.
		The \emph{partner of a vertex $v$ under $M$} is the unique vertex
		$u$ such that $uv\in M$.
		For distinct $i,j\in[0,3]$, let $x_{ij}=x_{ji}$ be the number of
		edges of $M$ with one end in $V_i$ and the other end in $V_j$.
		Let $\ell_i$ be the number of vertices of $V_i$ whose partner
		also lies in $V_i$; thus $\ell_i$ is twice the number of edges of
		$M$ with both ends in $V_i$.
		By~\eqref{eq:fourclasses},

		\begin{equation}\label{eq:matching-counts}
			\begin{aligned}
				\ell_0+x_{01}+x_{02}+x_{03}&=a,\\
				\ell_1+x_{01}+x_{12}+x_{13}&=b,\\
				\ell_2+x_{02}+x_{12}+x_{23}&=a,\\
				\ell_3+x_{03}+x_{13}+x_{23}&=b.
			\end{aligned}
		\end{equation}
		By~\eqref{eq:rank} and~\eqref{eq:t-def}, for $v\in V_i$ we have $t(v)=\alpha+\rho_i-ir+\beta\mathbf1_{\{i\ge2\}}+\rho(v)$, so for $u\in V_i$ and $v\in V_j$ the difference $t(v)-t(u)$ depends only on $i$, $j$, $\rho(u)$, and $\rho(v)$.

		We order each class as follows.  Partition $V_i$ into four \emph{blocks} according to the class containing the partner under $M$, and list the blocks in some \emph{partner order}.  For two distinct classes, list the edges of $M$ between them in the same order in the two corresponding blocks.  Within the block of $V_i$ whose partners lie in $V_i$, place the two ends of each edge of $M$ consecutively.  An edge of $M$ with both ends in $V_i$ then satisfies~\eqref{eq:matching-diff} automatically, so we only need to consider edges of $M$ between distinct classes; we say that such an edge is of \emph{type $ij$} if it joins $V_i$ and $V_j$ with $i<j$.

		We first use the partner order $0,1,2,3$ in every class.
		For an edge $uv\in M$ of type $ij$, with $u\in V_i$ and $v\in V_j$,
		let $p$ be its position, starting from $0$, in the common order of
		the edges between $V_i$ and $V_j$.
		Both ends then have position $p$ within their respective partner
		blocks.
		Their positions in the full classes are obtained by adding the
		sizes of the blocks that precede these partner blocks.
		In the difference $t(v)-t(u)$, the common terms $\alpha$ and $p$
		cancel.
		The resulting differences depend only on the edge type and are
		as follows; we derive each entry below.

		\begin{equation}\label{eq:orderA-table}
			\begin{array}{c|c}
				(i,j)&t(v)-t(u)\\ \hline
				01&a-r-\ell_0\\
				02&a+b-2r+\beta-\ell_0-x_{01}\\
				03&a+b-3r+\beta+x_{03}\\
				12&b-r+\beta+x_{02}-x_{01}-\ell_1\\
				13&a+b-2r+\beta-\ell_3-x_{23}\\
				23&b-r-\ell_3.
			\end{array}
		\end{equation}

		For type $01$, the partner block begins after the $\ell_0$
		vertices with partners in $V_0$ in the order on $V_0$, and it is
		first in the order on $V_1$.
		Thus $\rho(u)=\ell_0+p$ and $\rho(v)=p$.
		Since $\rho_1-\rho_0=a$, we obtain
		$t(v)-t(u)=a-r+(p-\ell_0-p)=a-r-\ell_0$.

		For type $02$, the partner block in $V_0$ is preceded by the
		blocks of sizes $\ell_0$ and $x_{01}$, while its corresponding
		block in $V_2$ is first.
		Hence $\rho(u)=\ell_0+x_{01}+p$ and $\rho(v)=p$.
		Now $\rho_2-\rho_0=a+b$, and only the vertex in $V_2$ receives
		the additional term $\beta$ in~\eqref{eq:rank}.
		Consequently,
		$t(v)-t(u)=a+b-2r+\beta-\ell_0-x_{01}$.

		For type $03$, the block in $V_0$ is last and the block in $V_3$
		is first.
		By~\eqref{eq:matching-counts},
		$\rho(u)=\ell_0+x_{01}+x_{02}+p=a-x_{03}+p$, whereas
		$\rho(v)=p$.
		Using $\rho_3-\rho_0=2a+b$, we obtain
		$t(v)-t(u)=2a+b-3r+\beta-(a-x_{03})
		=a+b-3r+\beta+x_{03}$.

		For type $12$, the block in $V_1$ is preceded by the blocks of
		sizes $x_{01}$ and $\ell_1$, and the block in $V_2$ is preceded
		by the block of size $x_{02}$.
		Thus $\rho(u)=x_{01}+\ell_1+p$ and $\rho(v)=x_{02}+p$.
		Since $\rho_2-\rho_1=b$, the difference is
		$t(v)-t(u)=b-r+\beta+x_{02}-x_{01}-\ell_1$.

		For type $13$, the block in $V_1$ is last, while the block in
		$V_3$ follows the block of size $x_{03}$.
		Hence $\rho(u)=x_{01}+\ell_1+x_{12}+p=b-x_{13}+p$ and
		$\rho(v)=x_{03}+p$.
		Using $\rho_3-\rho_1=a+b$, we obtain
		$t(v)-t(u)=a+b-2r+\beta+x_{03}+x_{13}-b$.
		The last equality in~\eqref{eq:matching-counts} gives
		$x_{03}+x_{13}-b=-\ell_3-x_{23}$, which yields the entry in
		the table.

		For type $23$, the block in $V_2$ is last, and the block in
		$V_3$ follows the blocks of sizes $x_{03}$ and $x_{13}$.
		Thus $\rho(u)=x_{02}+x_{12}+\ell_2+p=a-x_{23}+p$ and
		$\rho(v)=x_{03}+x_{13}+p=b-\ell_3-x_{23}+p$.
		Here $\rho_3-\rho_2=a$, and the two terms $\beta$ cancel.
		Therefore,
		$t(v)-t(u)=a-r+(b-\ell_3-x_{23})-(a-x_{23})
		=b-r-\ell_3$.

		We check that every entry of~\eqref{eq:orderA-table} lies in $[-r,r]$, except that the entry of type $12$ might exceed $r$.  From~\eqref{eq:ord-hyp} we have $a\ge r$, $b\ge r$, $a+b\ge2r$, and $a\le b\le 2r$, where the last bound follows from $2b-1\le a+b\le 3r$; and from~\eqref{eq:matching-counts}, each of $\ell_0+x_{01}+x_{02}$, $\ell_0+x_{01}$, $x_{03}$ is at most $a$, and each of $x_{01}+\ell_1$, $\ell_3+x_{23}$, $\ell_3$ is at most $b$.

		For type $01$, we have $-r\le a-r-\ell_0\le a-r\le r$, as $\ell_0\le a\le2r$.  For type $02$, the entry is at most $a+b+\beta-2r\le r$ by~\eqref{eq:ord-hyp}, and at least $(a+b+\beta-2r)-a=b+\beta-2r\ge-r$, as $\ell_0+x_{01}\le a$ and $b\ge r$.

		For type $03$, the entry is at least $a+b-3r\ge-r$, and at most $2a+b+\beta-3r\le r$ by~\eqref{eq:ord-hyp}, as $x_{03}\le a$.  For type $12$, the entry is at least $\beta+x_{02}-r\ge-r$, as $x_{01}+\ell_1\le b$; its upper bound is the one that may fail.

		For type $13$, the entry is at most $a+b+\beta-2r\le r$, and at least $a+\beta-2r\ge-r$, as $\ell_3+x_{23}\le b$ and $a\ge r$.  For type $23$, we have $-r\le b-r-\ell_3\le b-r\le r$, as $\ell_3\le b\le2r$.

		The only possible obstruction in the initial order is that the
		type-$12$ difference may exceed $r$. We therefore consider changing
		the relative positions of the blocks in $V_2$ whose partners lie in
		$V_0$ and $V_1$. To determine when such a change is needed and when
		it suffices, let
		\[
		A_0=3r-a-b-\beta+\ell_0+x_{01}
		\qquad\text{and}\qquad
		A_1=2r-b-\beta+\ell_1+x_{01}.
		\]
		The initial type-$02$ and type-$12$ differences are then $r-A_0$
		and $r+x_{02}-A_1$, respectively. Thus $A_0$ is the amount by which
		the type-$02$ difference may increase without exceeding $r$, while
		the type-$12$ difference is at most $r$ precisely when
		$x_{02}\le A_1$. Interchanging these two blocks in $V_2$ decreases
		the type-$12$ difference by $x_{02}$ and increases the type-$02$
		difference by $x_{12}$, so the latter remains at most $r$ precisely
		when $x_{12}\le A_0$. These observations motivate the first two
		cases below.

		Both $A_0$ and $A_1$ are nonnegative: $A_0\ge0$
		by~\eqref{eq:ord-hyp}, and $A_1\ge0$ because
		$b+\beta\le3r-a\le2r$ by~\eqref{eq:ord-hyp}.
		We consider three cases.

		\medskip
		\noindent\textbf{Case 1.} \emph{$x_{02}\le A_1$.}\par\nobreak
		\medskip
		The type-$12$ difference is $r+x_{02}-A_1\le r$.
		Since all other required bounds have already been verified,
		the initial order satisfies~\eqref{eq:matching-diff}.

		\medskip
		\noindent\textbf{Case 2.}
		\emph{$x_{02}>A_1$ and $x_{12}\le A_0$.}\par\nobreak
		\medskip
		We change only the partner order in $V_2$, to $1,0,2,3$.
		The type-$12$ difference becomes
		$b-r+\beta-x_{01}-\ell_1=r-A_1\le r$.
		It is also at least $\beta-r\ge-r$, since
		$x_{01}+\ell_1\le b$.
		The type-$02$ difference becomes
		$a+b-2r+\beta+x_{12}-\ell_0-x_{01}
		=r-A_0+x_{12}\le r$.
		Its lower bound remains valid because it has only increased.
		All other differences are unchanged.
		Thus~\eqref{eq:matching-diff} holds.

		\medskip
		\noindent\textbf{Case 3.}
		\emph{$x_{02}>A_1$ and $x_{12}>A_0$.}\par\nobreak
		\medskip
		Neither of the preceding two orders gives all the required bounds,
		so we split the relevant blocks. Let $k=A_1$ and $d=x_{12}-A_0$.
		In $V_2$, we keep the first $k$ vertices with partners in $V_0$
		before the type-$12$ block and place the remaining $x_{02}-k$
		vertices immediately after it. Each $V_2$-end of a type-$12$ edge
		then moves $x_{02}-k$ positions earlier, so its difference becomes
		$r+x_{02}-A_1-(x_{02}-k)=r$. However, the $V_2$-end of each of the
		remaining type-$02$ edges moves $x_{12}$ positions later. Without
		changing $V_0$, its difference would therefore increase from
		$r-A_0$ to $r-A_0+x_{12}=r+d$.

		To remove this excess of $d$, we split the corresponding type-$02$
		block in $V_0$ after its first $k$ vertices and insert the first
		$d$ vertices with partners in $V_3$ between the two parts.
		This moves the $V_0$-end of each of the remaining type-$02$ edges
		$d$ positions later, reducing its difference by $d$ to $r$.
		The first $k$ type-$02$ edges retain the positions of both ends,
		and the orders in $V_1$ and $V_3$ are unchanged.

		We first check that these choices are possible.
		We have $0\le k<x_{02}$ and $d>0$.
		By the third equality in~\eqref{eq:matching-counts},
		$x_{02}+x_{12}\le a$. Using the first equality there and the
		definition of $A_0$, we obtain
		\[
		x_{03}-d
		=3r-b-\beta-(x_{02}+x_{12})
		\ge3r-a-b-\beta
		\ge0.
		\]
		Thus $d\le x_{03}$. We use the following partner sequences in
		$V_0$ and $V_2$, where $j^m$ denotes $m$ consecutive vertices
		whose partners lie in $V_j$:
		\begin{align*}
			V_0:&\quad
			0^{\ell_0},\ 1^{x_{01}},\ 2^k,\ 3^d,\
			2^{x_{02}-k},\ 3^{x_{03}-d},\\
			V_2:&\quad
			0^k,\ 1^{x_{12}},\ 0^{x_{02}-k},\
			2^{\ell_2},\ 3^{x_{23}}.
		\end{align*}
		Figure~\ref{fig:case3-block-orders} compares these sequences with
		the initial orders.

		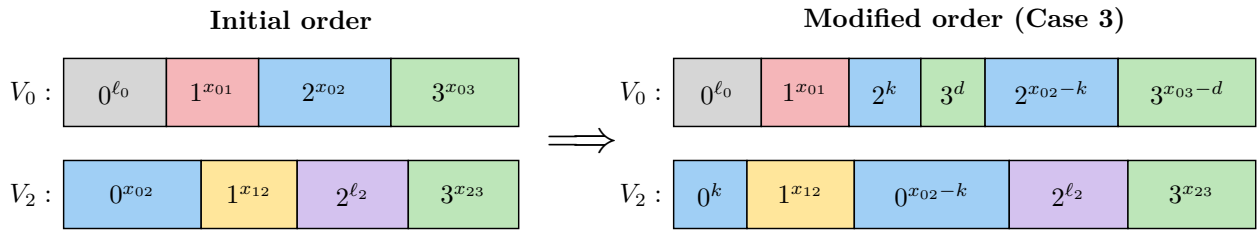
\begin{figure}[htbp]
			\centering
			\definecolor{tvordergray}{HTML}{D6D6D6}
			\definecolor{tvorderpink}{HTML}{F5B6B9}
			\definecolor{tvorderblue}{HTML}{A0CFF5}
			\definecolor{tvordergreen}{HTML}{BDE6B9}
			\definecolor{tvorderyellow}{HTML}{FFE69B}
			\definecolor{tvorderpurple}{HTML}{D4C2EF}
			\newcommand{\tvorderblock}[5]{%
				\path[draw=black,line width=.55pt,fill=#4]
				(#1,#2) rectangle ++(#3,.90);
				\node[inner sep=0pt,text=black]
				at ({#1+.5*#3},{#2+.45}) {$#5$};}
			\begin{tikzpicture}[x=1cm,y=1cm,font=\small]
				\node[font=\small\bfseries,inner sep=0pt] at (3.69,2.76)
				{Initial order};
				\node[font=\small\bfseries,inner sep=0pt] at (12.60,2.76)
				{Modified order (Case 3)};

				\node[anchor=east,inner sep=0pt] at (.53,1.80) {$V_0:$};
				\tvorderblock{.68}{1.35}{1.36}{tvordergray}{0^{\ell_0}}
				\tvorderblock{2.04}{1.35}{1.22}{tvorderpink}{1^{x_{01}}}
				\tvorderblock{3.26}{1.35}{1.75}{tvorderblue}{2^{x_{02}}}
				\tvorderblock{5.01}{1.35}{1.69}{tvordergreen}{3^{x_{03}}}

				\node[anchor=east,inner sep=0pt] at (.53,.45) {$V_2:$};
				\tvorderblock{.68}{0}{1.82}{tvorderblue}{0^{x_{02}}}
				\tvorderblock{2.50}{0}{1.27}{tvorderyellow}{1^{x_{12}}}
				\tvorderblock{3.77}{0}{1.47}{tvorderpurple}{2^{\ell_2}}
				\tvorderblock{5.24}{0}{1.46}{tvordergreen}{3^{x_{23}}}

				\node[font=\LARGE,inner sep=0pt] at (7.54,1.125) {$\Longrightarrow$};

				\node[anchor=east,inner sep=0pt] at (8.60,1.80) {$V_0:$};
				\tvorderblock{8.75}{1.35}{1.16}{tvordergray}{0^{\ell_0}}
				\tvorderblock{9.91}{1.35}{1.16}{tvorderpink}{1^{x_{01}}}
				\tvorderblock{11.07}{1.35}{.95}{tvorderblue}{2^k}
				\tvorderblock{12.02}{1.35}{.85}{tvordergreen}{3^d}
				\tvorderblock{12.87}{1.35}{1.76}{tvorderblue}{2^{x_{02}-k}}
				\tvorderblock{14.63}{1.35}{1.82}{tvordergreen}{3^{x_{03}-d}}

				\node[anchor=east,inner sep=0pt] at (8.60,.45) {$V_2:$};
				\tvorderblock{8.75}{0}{.97}{tvorderblue}{0^k}
				\tvorderblock{9.72}{0}{1.42}{tvorderyellow}{1^{x_{12}}}
				\tvorderblock{11.14}{0}{2.05}{tvorderblue}{0^{x_{02}-k}}
				\tvorderblock{13.19}{0}{1.57}{tvorderpurple}{2^{\ell_2}}
				\tvorderblock{14.76}{0}{1.69}{tvordergreen}{3^{x_{23}}}
			\end{tikzpicture}
			\caption{The block rearrangement in Case~3, where $k=A_1$ and
				$d=x_{12}-A_0$. A block $j^m$ consists of $m$ consecutive vertices
				whose partners lie in $V_j$. }
			\label{fig:case3-block-orders}
		\end{figure}

		In particular, $2^{\ell_2}$ in the $V_2$ row consists of the
		$\ell_2$ vertices of $V_2$ whose partners also lie in $V_2$.
		The positions of all vertices in this block are unchanged.
		We keep $V_1$ and $V_3$ in their original orders and list the
		edges of $M$ consistently in corresponding blocks. In particular,
		the first $k$ type-$02$ edges in $V_0$ are also the first $k$
		type-$02$ edges in $V_2$, and the first $d$ type-$03$ edges in
		$V_0$ are also the first $d$ type-$03$ edges in $V_3$.

		The differences for types $01$, $13$, and $23$ are unchanged,
		as are those for the first $k$ type-$02$ edges and the last
		$x_{03}-d$ type-$03$ edges. Moreover, the two ends of every edge
		$uv\in M$ within a class remain consecutive, so
		$|t(u)-t(v)|=1\le r$. Thus it remains to check the differences
		for the remaining $x_{02}-k$ type-$02$ edges, all type-$12$ edges,
		and the first $d$ type-$03$ edges. As in~\eqref{eq:orderA-table}, for
		an edge of type $ij$ we write $u$ for its end in $V_i$ and $v$ for its
		end in $V_j$, where $i<j$, and we compute $t(v)-t(u)$.

		For each of the remaining type-$02$ edges $uv$, with
		$u\in V_0$ and $v\in V_2$, using $d=x_{12}-A_0$ and substituting
		the definition of $A_0$ gives
		\[
		\begin{aligned}
			t(v)-t(u)
			&=a+b-2r+\beta-\ell_0-x_{01}+x_{12}-d\\
			&=a+b-2r+\beta-\ell_0-x_{01}+A_0\\
			&=a+b-2r+\beta-\ell_0-x_{01}+(3r-a-b-\beta+\ell_0+x_{01})\\
			&=r.
		\end{aligned}
		\]
		For every type-$12$ edge $uv$, with $u\in V_1$ and $v\in V_2$,
		using $k=A_1$ and substituting the definition of $A_1$ gives
		\[
		\begin{aligned}
			t(v)-t(u)
			&=b-r+\beta-x_{01}-\ell_1+k\\
			&=b-r+\beta-x_{01}-\ell_1+A_1\\
			&=b-r+\beta-x_{01}-\ell_1 +(2r-b-\beta+\ell_1+x_{01})\\
			&=r.
		\end{aligned}
		\]

		For each of the first $d$ type-$03$ edges $uv$, with
		$u\in V_0$ and $v\in V_3$, the modified order gives
		\[
		t(v)-t(u)
		=2a+b-3r+\beta-\ell_0-x_{01}-k.
		\]
		By the first equality in~\eqref{eq:matching-counts},
		$\ell_0+x_{01}+x_{02}+x_{03}=a$, and hence
		\[
		\begin{aligned}
			t(v)-t(u)
			&=(a+b-3r+\beta+x_{03})
			+(a-\ell_0-x_{01}-x_{03}-k)\\
			&=(a+b-3r+\beta+x_{03})+(x_{02}-k)\\
			&\ge-r,
		\end{aligned}
		\]
		where the last inequality follows because the original
		type-$03$ difference in~\eqref{eq:orderA-table} is at least
		$-r$ and $x_{02}-k>0$.
		For the upper bound, substituting
		$k=A_1=2r-b-\beta+\ell_1+x_{01}$ gives
		\[
		\begin{aligned}
			r-\bigl(t(v)-t(u)\bigr)
			&=4r-2a-b-\beta+\ell_0+x_{01}+k\\
			&=4r-2a-b-\beta+\ell_0+x_{01} +(2r-b-\beta+\ell_1+x_{01})\\
			&=6r-2a-2b-2\beta+\ell_0+\ell_1+2x_{01}\\
			&=2(3r-a-b-\beta)+\ell_0+\ell_1+2x_{01}\\
			&\ge0,
		\end{aligned}
		\]
		where the last inequality follows from
		$a+b+\beta\le3r$ in~\eqref{eq:ord-hyp} and the nonnegativity
		of $\ell_0,\ell_1$, and $x_{01}$. Therefore
		$-r\le t(v)-t(u)\le r$ for these type-$03$ edges.

		Consequently, $|t(u)-t(v)|\le r$ for every $uv\in M$,
		and~\eqref{eq:matching-diff} holds in this case as well.
		This completes the proof of Claim~\ref{clm:ordering}.\qedhere

	\end{claimproof}

	\section{Concluding remarks}\label{sec:remarks}

	The proof of Theorem~\ref{thm:cubic} assigns disjoint intervals of weights to the degree classes of a balanced spanning subgraph. For $4$-regular graphs, we first reduce to a cubic graph by deleting a maximum matching and smoothing the uncovered vertices. The labels of the virtual edges determine the intervals reserved for the uncovered vertices, while the ordering in Claim~\ref{clm:ordering} allows the matched vertices to attain the remaining weights.

	Theorems~\ref{thm:cubic} and~\ref{thm:four-main} show that the bound in~\eqref{eq:regular-lower} is attained by every cubic or $4$-regular graph. For each fixed $d\ge2$, it is also attained by all sufficiently large $d$-regular graphs. This is a consequence of the following prescribed degree-frequency theorem of Cao, Tang, and Wu~\cite{CTW}.

	\begin{theorem}[Cao, Tang, and Wu~{\cite[Theorem~2.1]{CTW}}]\label{thm:CTW}
		For every $\varepsilon>0$ there exists $n_0=n_0(\varepsilon)$ such that the following holds.  Let $G$ be a $d$-regular graph on $n\ge n_0$ vertices, where $2\le d\le n^{1/12-\varepsilon}$, and let $(q_0,\ldots,q_d)$ be a vector of nonnegative integers such that
		\[
		\sum_{i=0}^{d} q_i=n,\qquad
		\sum_{i=0}^{d} iq_i\equiv0\pmod 2,\qquad\text{and}\qquad
		\left|q_i-\frac{n}{d+1}\right|\le1\ \text{ for each } i\in[0,d].
		\]
		Then $G$ has a spanning subgraph $H$ with $n_i(H)=q_i$ for each $i\in[0,d]$.
	\end{theorem}

	\begin{theorem}\label{thm:large-d}
		Let $d\ge2$ be an integer.  If $G$ is a $d$-regular graph on $n$ vertices and $n$ is sufficiently large as a function of $d$, then $\tvs(G)=\ceil{(n+d)/(d+1)}$.
	\end{theorem}

	Theorem~\ref{thm:large-d} applies to disconnected graphs and does not require a perfect matching or a $1$-factorization.

	\begin{proof}

		The lower bound is~\eqref{eq:regular-lower}. Let
		$s=\ceil{\frac{n+d}{d+1}}$ and
		$r=s-1=\ceil{\frac{n-1}{d+1}}$.

		We first show that if $G$ has a spanning subgraph $H$ with
		$n_i(H)\le r$ for every $i\in[0,d]$, with at most one exception
		$j\in[0,d]$ for which $n_j(H)=r+1$, then $\tvs(G)\le s$.
		For each $i\in[0,d]$, let $q_i=n_i(H)$ and
		$V_i(H)=\{v_{i,0},\ldots,v_{i,q_i-1}\}$, and let $j=d+1$ if
		no class has size $r+1$.

		Label every edge of $H$ by $s$ and every other edge of $G$ by $1$.
		A vertex in $V_i(H)$ is incident with $i$ edges of $H$ and $d-i$
		other edges of $G$, so its incident edge labels sum to
		$is+(d-i)=d+ir$. Its possible weights therefore form the interval
		$I_i=[d+1+ir,\ d+1+(i+1)r]$ as its vertex label ranges over $[1,s]$.
		Consecutive intervals $I_i$ and $I_{i+1}$ share exactly one point,
		namely $d+1+(i+1)r$, and nonconsecutive intervals are disjoint.
		A class of at most $r$ vertices can avoid one endpoint of its
		interval, whereas the exceptional class of $r+1$ vertices must
		use the entire interval. We therefore assign labels starting
		from $1$ in the classes up to and including $V_j(H)$, and labels
		starting from $2$ in the classes after $V_j(H)$. The first group
		of nonexceptional classes will avoid the upper endpoints of their
		intervals, and the second group will avoid the lower endpoints.
		If $j=d+1$, there is no exceptional class and all labels start
		from $1$.

		More precisely, let
		\[
		\lambda(v_{i,\ell})=1+\ell+\mathbf1_{\{i>j\}}
		\qquad\text{for }i\in[0,d]\text{ and }\ell\in[0,q_i-1].
		\]
		Thus the term $1+\ell$ gives distinct consecutive labels within
		each class, while $\mathbf1_{\{i>j\}}$ shifts all labels in the
		classes after the exceptional class up by one. If $i<j$, the
		labels are $1,\ldots,q_i$ with $q_i\le r$; if $i=j$, they are
		$1,\ldots,r+1$; and if $i>j$, they are $2,\ldots,q_i+1$ with
		$q_i\le r$. Hence all labels lie in $[1,r+1]=[1,s]$. For every
		$v_{i,\ell}\in V_i(H)$, we have
		\[
		\wt_{\lambda}(v_{i,\ell})
		=d+ir+\lambda(v_{i,\ell})
		=d+1+ir+\ell+\mathbf1_{\{i>j\}}.
		\]

		Within each class $V_i(H)$ the weights are distinct and lie in
		$I_i$. If $i<j$, then $\ell\le q_i-1\le r-1$, so the weights are
		at most $d+(i+1)r$, one less than the upper endpoint of $I_i$.
		If $i>j$, the added $1$ makes every weight at least $d+2+ir$,
		one greater than the lower endpoint of $I_i$; also
		$\ell+1\le q_i\le r$, so the weights do not exceed the upper
		endpoint. If $j\le d$, then $q_j=r+1$ and the weights in $V_j(H)$
		fill $I_j$, including both endpoints. Thus only the exceptional
		class uses both endpoints of its interval. For each $i\in[0,d-1]$,
		the point shared by $I_i$ and $I_{i+1}$ is avoided by $V_i(H)$
		when $i<j$, and by $V_{i+1}(H)$ when $i\ge j$.
		No two classes therefore have a common weight. Hence $\lambda$
		is vertex irregular, and $\tvs(G)\le s$.

		It remains to find such a spanning subgraph $H$ when $n$ is large.
		We will first prescribe nearly equal degree-class sizes
		$q_0,\ldots,q_d$ and then apply Theorem~\ref{thm:CTW} to obtain $H$.
		Besides summing to $n$, the prescribed sizes must satisfy
		$\sum_{i=0}^{d}iq_i\equiv0\pmod2$, since this sum will equal
		$2|E(H)|$. Write $n=(d+1)a+b$ with $0\le b\le d$.
		Then $r=a$ if $b\le1$, and $r=a+1$ if $b\ge2$.
		The idea is to begin with $a$ vertices in each class and distribute
		the remaining $b$ vertices among distinct classes, choosing these
		classes to make the degree sum even. When $b=0$, we instead
		transfer one vertex between two classes if a parity correction
		is needed.

		Suppose first that $b>0$. Let $S_0=[0,b-1]$ and
		$S_1=(S_0\setminus\{b-1\})\cup\{b\}$. Both are $b$-element
		subsets of $[0,d]$, and replacing $b-1$ by $b$ increases the sum
		of their elements by exactly one. Thus exactly one choice
		$S\in\{S_0,S_1\}$ makes
		$a\binom{d+1}{2}+\sum_{i\in S}i$ even. Choose this $S$, and let
		$q_i=a+1$ for $i\in S$ and $q_i=a$ for $i\in[0,d]\setminus S$.
		Then $\sum_{i=0}^{d}q_i=(d+1)a+b=n$. The initial $a$ vertices
		in each degree-$i$ class contribute $ia$ to the prescribed degree
		sum, and the additional vertex in each class indexed by $S$
		contributes $i$. Hence
		\[
		\sum_{i=0}^{d}iq_i
		=a\sum_{i=0}^{d}i+\sum_{i\in S}i
		=a\binom{d+1}{2}+\sum_{i\in S}i,
		\]
		which is even by the choice of $S$.
		Also $|q_i-n/(d+1)|\le1$ for each $i\in[0,d]$, since
		$n/(d+1)=a+b/(d+1)$. If $b=1$, then $r=a$ and exactly one class
		has size $r+1$; if $b\ge2$, then $r=a+1$ and every class has
		size at most $r$.

		Suppose next that $b=0$, so $r=a$. Let $q_i=a$ for every
		$i\in[0,d]$. If $\sum_{i=0}^{d}iq_i=a\binom{d+1}{2}$ is odd,
		replace $q_0$ and $q_1$ by $a-1$ and $a+1$, respectively.
		This transfers one vertex from the degree-$0$ class to the
		degree-$1$ class, increasing the prescribed degree sum by one
		and making it even. It preserves $\sum_{i=0}^{d}q_i=n$ and
		$|q_i-n/(d+1)|\le1$, and creates exactly one class of size $r+1$.
		The entries remain nonnegative because $a\ge1$.
		If no correction is needed, every class has size $r$.

		In either case, $(q_0,\ldots,q_d)$ satisfies the numerical
		conditions of Theorem~\ref{thm:CTW} and the class-size bounds
		required in the first part of the proof. Fix $\varepsilon=1/24$.
		If $n$ is sufficiently large as a function of $d$, then
		$n\ge n_0(\varepsilon)$ and $d\le n^{1/12-\varepsilon}$, so
		Theorem~\ref{thm:CTW} gives a spanning subgraph $H$ of $G$ with
		$n_i(H)=q_i$ for each $i\in[0,d]$. By the first part of the proof,
		$\tvs(G)\le s$. Together with~\eqref{eq:regular-lower}, this gives
		$\tvs(G)=s$, as desired.\qedhere

	\end{proof}

	For $d=3$ and $d=4$, Theorem~\ref{thm:large-d} gives the equalities in Theorems~\ref{thm:cubic} and~\ref{thm:four-main} only for sufficiently large $n$. We conclude with the restriction of Conjecture~\ref{conj:NBSG} to regular graphs.

	\begin{conjecture}\label{conj:regular}
		Let $G$ be a $d$-regular graph on $n$ vertices. Then
		\[
		\tvs(G)=\ceil{\frac{n+d}{d+1}}.
		\]
	\end{conjecture}

	To our knowledge, for each fixed $d\ge5$, the conjecture remains open without a restriction on $n$. It would be interesting to determine whether the degree-balanced decompositions and matching reductions used here extend to these degrees.

	\section*{Declaration on the Use of AI Tools}

	In the course of this work the authors used generative AI tools.  For the cubic case, ChatGPT~Pro  located the paper of Lu\v{z}ar, Przyby\l{}o, and Sot\'ak~\cite{LPS}, on which the proof of Theorem~\ref{thm:cubic} relies.  For the $4$-regular case, the authors provided a framework of proof ideas, and ChatGPT assisted in working out the details of the proof of Theorem~\ref{thm:four-main}.  The exposition was revised with the assistance of Claude  and ChatGPT.  The authors have checked all arguments and take full responsibility for the content of this paper.

\end{document}